\documentclass[1p, 12pt]{elsarticle}
\usepackage{amsmath, amssymb, amsthm, bbm, mathtools, commath, amsfonts}
\usepackage{hyperref}
\usepackage{pslatex}
\usepackage{verbatim}
\usepackage{lmodern}
\usepackage{enumitem}

\makeatletter
\def\ps@pprintTitle{%
    \let\@oddhead\@empty
    \let\@evenhead\@empty
    \def\@oddfoot{\footnotesize\itshape
         {} \hfill}%
    \let\@evenfoot\@oddfoot
    }
\makeatother

\newcommand{\CC}{\mathbb C}
\newcommand{\NN}{\mathbb N}
\newcommand{\RR}{\mathbb R}
\newcommand{\ZZ}{\mathbb Z}
\newcommand{\mc}[1]{\mathcal{#1}}
\newcommand{\mrm}[1]{\mathrm{#1}}
\newcommand{\ol}[1]{\overline{#1}}

\newtheorem{theorem}{Theorem}[section]
\newtheorem{lemma}[theorem]{Lemma}
\newtheorem{corollary}[theorem]{Corollary}
\newtheorem{proposition}[theorem]{Proposition}

\theoremstyle{definition}

\newtheorem{observation}[theorem]{Observation}
\newtheorem{definition}[theorem]{Definition}

\newcommand{\GG}{\mathcal{G}}
\newcommand{\reg}{\rho_{\text{reg}}}
\newcommand{\Vreg}{V_{\text{reg}}}
\DeclareMathOperator{\Co}{Co}
\DeclareMathOperator{\Tr}{Tr}
\DeclareMathOperator{\HS}{HS}

\begin{document}

\title{Spectra of Convex Hulls of Matrix Groups}
\author{Eric Jankowski\corref{cor1}}            
\ead{ejankowski@berkeley.edu}
\address{University of California, Berkeley, United States of America}
\cortext[cor1]{Corresponding author}

\author{Charles R. Johnson}
\ead{crjohn@wm.edu}
\address{College of William and Mary, United States of America}

\author{Derek Lim}
\ead{dl772@cornell.edu}
\address{Cornell University, United States of America}

\begin{abstract}
The still-unsolved problem of determining the set of eigenvalues realized by $n$-by-$n$ doubly stochastic matrices, those entrywise-nonnegative matrices with row sums and column sums equal to $1$, has attracted much attention in the last century. This problem is somewhat algebraic in nature, due to a result of Birkhoff demonstrating that the set of doubly stochastic matrices is the convex hull of the permutation matrices. Here we are interested in a general matrix group $G \subseteq GL_n(\CC)$ and the \emph{hull spectrum} $\HS(G)$ of eigenvalues realized by convex combinations of elements of $G$. We show that hull spectra of matrix groups share many nice properties. Moreover, we give bounds on the hull spectra of matrix groups, determine $\HS(G)$ exactly for important classes of matrix groups, and study the hull spectra of representations of abstract groups.
\end{abstract}


\maketitle

\section{Introduction}

Our study of this subject was originally motivated by the work done over the last century in investigating $DS_n$, the subset of the complex plane consisting of every complex number that is an eigenvalue of some $n$-by-$n$ doubly stochastic matrix. This problem is part of a larger class of problems related to the Nonnegative Inverse Eigenvalue Problem, which asks which multisets of $n$ complex numbers are achievable as the spectrum of some matrix in $M_n(\RR)$ with nonnegative entries. The analogous problem for row stochastic matrices was solved in 1951 \cite{Karpelevich}. In \cite{PerfectMirsky}, Perfect and Mirsky proved some properties of $DS_n$, including that $DS_n = \bigcup_{k \leq n} \Pi_k$ for $n \leq 3$, where $\Pi_k$ denotes the convex hull of the $k$th roots of unity. More recently, it was shown in \cite{DS4} that $DS_4 = \bigcup_{k \leq 4} \Pi_k$. However, there is an example of a doubly stochastic matrix in \cite{MR} that shows that $DS_5$ contains some points that do not belong to $\bigcup_{k \leq 5}\Pi_k$. Thus, a precise description of $DS_n$ remains elusive for $n \geq 5$.

Birkhoff's theorem \cite{Birkhoff} states that the set of doubly stochastic matrices of size $n$ is precisely the convex hull of the permutation matrices of size $n$---a specific representation of the symmetric group $S_n$. Here, we study the region analogous to $DS_n$ for different matrix groups---given a matrix group $G$ we consider its \textit{hull spectrum} $\HS(G)$, which is the subset of the complex plane consisting of points achieved as eigenvalues of matrices in the convex hull of $G$. Moreover, for an abstract group $\GG$ we study the hull spectra $\HS(\rho(\GG))$ for representations $\rho$ of $\GG$.

We show that hull spectra of finite matrix groups share many of the nice properties that $DS_n$ possesses. Namely, they are contained in the unit disc, star-shaped from any $\eta \in [0,1]$ for nontrivial groups, and bounded above and below by some natural regions related to the group. These properties are notable, since the hull spectrum of a finite set of matrices $S \subset M_n(\CC)$ is not generally well-behaved. For instance, it is not difficult to construct sets $S$ for which $\HS(S)$ achieves values of larger magnitude than the spectral radii of matrices in $S$, or for which $\HS(S)$ is not connected, let alone star-shaped.

 In studying the hull spectra of specific classes of groups, we first classify the hull spectra of (possibly infinite) abelian matrix groups. Applying these results to cyclic subgroups of a general matrix group $G$, we attain lower bounds on the hull spectrum of $G$. We show that for some classes of groups, the lower bound is always attained, so the hull spectra are in fact unions of certain polygons $\Pi_k$. However, we note that there are groups that have hull spectra with exceptional regions outside of these polygons. Finally, we apply our results to $DS_n$, and thus prove properties of $DS_n$ without using the theory of nonnegative matrices or specific facts about doubly-stochastic matrices, which have been used in the past to study $DS_n$. We thus find that the concept of the hull spectra of a matrix group is useful in the study of $DS_n$, and is also an interesting concept in its own right with much nice structure to explore.

\section{Preliminaries and Basic Properties}

\begin{definition}
Let $S \subseteq M_n(\CC)$ be a set of $n$-by-$n$ complex matrices. The \emph{convex hull} of $S$, denoted $\Co(S)$, is the set of all convex combinations of elements of $S$. The \emph{hull spectrum} of $S$, denoted $\HS(S)$, is the set of all complex numbers that occur as eigenvalues of some matrix in $\Co(S)$. In most cases, we will consider the hull spectrum $\HS(G)$ of a matrix group $G \subseteq GL_n(\CC)$.
\end{definition}

Given an abstract group $\GG$, the hull spectra of different $\GG$-representations may vary significantly. For instance, the trivial representation has hull spectrum $\{1\}$. Since a non-faithful representation may be viewed as a representation of a particular quotient group, we primarily consider the hull spectra of faithful representations. This restriction allows us to give tighter bounds (on hull spectra) that better characterize the group.

\begin{definition}
Given an abstract group $\GG$, its \emph{faithful hull spectra} are the hull spectra $\HS(\rho(\GG))$ for faithful $\GG$-representations $\rho$ (if they exist), in which we write $\rho(\GG)$ to denote the image of $\GG$ under $\rho$.
\end{definition}

We now present some relatively straightforward lemmas that will aid us in later proofs.

\begin{lemma}\label{SubgroupLemma}
If $H$ is a matrix subgroup of $G$, then $\HS(H) \subseteq \HS(G)$.
\end{lemma}

\begin{proof}
We have $\Co(H) \subseteq \Co(G)$, so $\HS(H) \subseteq \HS(G)$.
\end{proof}

\begin{lemma}\label{IsomorphicRepresentationLemma}
Let $\GG$ be a group, and let $\rho_1, \rho_2$ be isomorphic $\GG$-representations. Then $\HS(\rho_1(\GG)) = \HS(\rho_2(\GG))$.
\end{lemma}
\begin{proof}
Denote by $n$ the degree of $\rho_1$ and $\rho_2$. Since $\rho_1$ and $\rho_2$ are isomorphic, there is $S \in GL_n(\CC)$ such that $\rho_1(g) = S\rho_2(g) S^{-1}$ for all $g \in \GG$. Then all convex combinations in $\Co(\rho_1(\GG))$ are similar to the corresponding convex combinations in $\Co(\rho_2(\GG))$ by $S$, so their eigenvalues coincide.
\end{proof}

\begin{lemma}\label{IrrepLemma}
Let $\GG$ be a finite group and $\rho : \GG \to GL(V)$ a finite-dimensional $\GG$-representation with irreducible direct sum decomposition $\rho = \bigoplus_{i=1}^m \rho_i$. Then $\HS(\rho(\GG)) = \bigcup_{i=1}^m \HS(\rho_i(\GG))$.
\end{lemma}
\begin{proof}
Since $\rho = \bigoplus_{i=1}^m \rho_i$, we may choose a basis of $V$ such that each element of $\rho(\GG)$ is block upper triangular with blocks $\rho_i(\GG)$ along the diagonal. Then the spectrum of any element of $\Co(\rho(\GG))$ is the union of the spectra of the corresponding elements of $\Co(\rho_i(\GG))$ for $i=1,...,m$.
\end{proof}

For a matrix group $G$, the geometry of $\HS(G)$ is now considered. In particular, we discuss the set of \emph{star points} of $\HS(G)$, or points $\lambda \in \HS(G)$ such that for any $\lambda' \in \HS(G)$, all points on the line segment between $\lambda$ and $\lambda'$ lie in $\HS(G)$. In this case, we say that $\HS(G)$ is \emph{star-shaped} from $\lambda$.

\begin{lemma}\label{StarConvexLemma}
Let $S \subseteq M_n$, and suppose $\Co(S)$ contains $\eta I$ for some $\eta \in \CC$. Then $\HS(S)$ is star-shaped from $\eta$. In particular, for a matrix group $G$, $\HS(G)$ is star-shaped from 1.
\end{lemma}

\begin{proof}
Let $\lambda$ be an eigenvalue of the convex combination $C \in \Co(S)$. Since $\eta I \in \Co(S)$, we have $(1-\alpha)C + \alpha \eta I \in \Co(S)$ and hence the eigenvalue $(1-\alpha)\lambda + \alpha \eta$ is in $\HS(S)$ for any $\alpha \in [0,1]$. Thus, the whole line segment between $\eta$ and $\lambda$ is contained in $\HS(S)$ Any matrix group $G$ contains the identity, and hence is star-shaped from $1$.
\end{proof}

\begin{lemma}\label{ContainsZeroLemma}
Let $\rho : \GG \to GL(V)$ be a nontrivial irreducible representation of the finite group $\GG$. Then the zero matrix belongs to $\Co(\rho(G))$. In particular, $\HS(\rho(\GG))$ is star-shaped from $0$.
\end{lemma}
\begin{proof}
Notice that the group average $A = \frac{1}{\abs{\GG}} \sum_{g \in \GG} \rho(g)$ is an idempotent element of $\Co(\rho(\GG))$. Using the well-known result (see for instance \cite{Serre}) that the sum of character values of a nontrivial irreducible representation satisfies $\sum_{g \in \GG} \chi(g) = 0$, we find that $\Tr(A) = 0$. Since $A$ is idempotent, it follows that $A = 0 \in \Co(\rho(\GG))$, as desired. Applying Lemma \ref{StarConvexLemma} shows that $\HS(\rho(\GG))$ is star-shaped from $0$.
\end{proof}

\begin{corollary}\label{StarCorollary}
Let $G$ be a nontrivial finite matrix group. Then $\HS(G)$ is star-shaped from every point in $[0,1]$.
\end{corollary}

\begin{proof}
View $G$ as a matrix group representation of some abstract group $\GG$. Since $G$ is nontrivial, the representation contains some nontrivial summands. Each of these summands has a hull spectrum that is star-shaped from both $0$ and $1$, so $\HS(G)$ is also star-shaped from $0$ and $1$. Since the set of star points of a subset of a vector space is convex, the corollary is proven.
\end{proof}

Another important geometric consideration is the relationship of $\HS(G)$ with the unit circle $S_1$ and the closed unit disc $D$. See \cite{MatrixAnalysis} for necessary results on matrix norms.

\begin{lemma}\label{UnitBallLemma}
Let $S \subseteq M_n$ be a collection of matrices, and suppose there is a matrix norm $\|\cdot\|$ such that $\|A\| \leq 1$ for all $A \in S$. Then $\HS(S)$ is contained in the closed unit disc $D$. In particular, we have $\HS(G) \subseteq D$ for any finite matrix group $G$.
\end{lemma}

\begin{proof}
This is clear by convexity of norms and the fact that matrix norms bound the spectral radius. For any convex combination of group elements $\sum_k \alpha_k A_k \in \Co(S)$, it holds that
\[ r \left(\sum_k \alpha_k A_k \right) \leq \norm{\sum_k \alpha_k A_k} \leq \sum_k \alpha_k \norm{A_k} = 1 \]

\noindent in which $r(A)$ is the spectral radius of $A$. If $G$ is a finite matrix group, we know as a basic result of representation theory that there is a change of basis such that each element of $G$ is unitary. By Lemma \ref{IsomorphicRepresentationLemma}, this change of basis does not affect the hull spectrum. Since the spectral norm of a unitary matrix is $1$, the claim follows.
\end{proof}

It is important that the condition of the previous lemma hold for finite matrix groups $G$. Consider the following matrices and note that the condition does not hold for $S_x$:
\[S_x = \Big\{\begin{bmatrix}1 & 2x \\ 0 & 1 \end{bmatrix}, \begin{bmatrix}1 & 0 \\ 2x & 1 \end{bmatrix}\Big\}, \qquad B_x = \begin{bmatrix}1 & x\\ x & 1 \end{bmatrix} \in \Co(S_x) \]
Since $B_x \in \Co(S_x)$, letting $x \to \infty$, we see that the $\HS(S_x)$ contain eigenvalues of arbitrarily high magnitude while each matrix in $S_x$ has only $1$ as its sole unique eigenvalue.

Now we present a result restricting the multisets of eigenvalues achievable by $C \in \Co(G)$ for a finite matrix group $G$.

\begin{proposition}\label{ExtremalEigenvalues}
For $G$ a finite matrix group, a matrix $C \in \Co(G)$ has all eigenvalues on the unit circle if and only if $C$ is an element of $G$.
\end{proposition}

\begin{proof}
Since $G$ is a finite matrix group, we can assume $A_i \in G$ are all unitary. Let $C = \sum_{i=1}^k \alpha_i A_i$ where $\alpha_i > 0$ and $\sum_i \alpha_i = 1$ be an element of $\Co(G)$, and suppose that $C$ has all eigenvalues on the unit circle. Then $1 = \abs{\det(C)} = \sigma_1 \cdots \sigma_n$ where $\sigma_i$ are the singular values of $C$. 
Due to the convexity of the spectral norm and the fact that $C$ is a convex combination of unitary matrices, which all have spectral norm $1$, we know that the largest singular value satisfies $\sigma_1 \leq 1$, so in fact all singular values of $C$ are $1$. This implies that $C$ is unitary. Then $A_1^* C$ is also unitary and of the form
\[A_1^*C  = \alpha_1 I + \alpha_2 A_1^* A_2 \ldots + \alpha_k A_1^* A_k\]
The eigenvalues of $A_1^*C$ are of the form $\alpha_1 + \mu_i$, where $\mu_i$ is an eigenvalue of $\sum_{i=2}^k \alpha_i A_1^* A_i$. Each $\mu_i$ has magnitude at most $\sum_{i=2}^k \alpha_i$, as we can again apply convexity of the spectral norm since the $A_1^* A_i$ are unitary. Moreover, since $A_1^*C$ is unitary, all of its eigenvalues have magnitude $1$. This forces the eigenvalues of $A_1^*C$ to all be $1$, since this is the case of equality in the triangle inequality $1 = \abs{\alpha_1 + \mu_i} \leq \alpha_1 + \sum_{i=2}^n \alpha_i = 1$. Thus, $A_1^*C = I$ and thus $C = A_1$.
\end{proof}

\section{Abelian Groups}\label{AbelianSection}

We begin by considering groups of $1$-by-$1$ matrices i.e. groups of complex numbers. As we will see in Theorem \ref{AbelianTheorem}, the hull spectrum problem for any abelian group of matrices may be reduced to the $1$-by-$1$ case.

For a subset $S \subseteq GL_n(\CC)$, denote by $\langle S \rangle$ the matrix group generated by $S$. Let $S_1$ and $B_1$ denote the unit circle and open unit disc of the complex plane respectively. Moreover, let $B_\alpha' = B_1 \cup \{e^{2\pi i \alpha n} : n \in \ZZ\}$, and let $\RR^+$ (resp.\ $\RR^-$) be the strictly positive (resp.\ negative) real numbers.

\begin{lemma}
Let $\lambda \in \CC \backslash \{0\}$. Then
\begin{enumerate}
    \item if $\lambda$ is a primitive $n$th root of unity, then $\HS(\langle \lambda \rangle) = \Pi_n$.
    
    \item if $\lambda = e^{2\pi i \alpha}$ is not a root of unity, then $\HS(\langle \lambda \rangle) = B_\alpha'$.
    
    \item if $\lambda \in \RR^+ \backslash\{1\}$, then $\HS(\langle \lambda \rangle) = \RR^+$.
    
    \item if $\lambda \in \RR^- \backslash\{-1\}$, then $\HS(\langle \lambda \rangle) = \RR$.
    
    \item if $\lambda \notin S_1 \cup \RR$, then $\HS(\langle \lambda \rangle) = \CC$.
\end{enumerate}
\end{lemma}

\begin{proof}
We address each case individually. Note that (1) and (2) correspond to the case in which $\lambda \in S_1$, and that (3) and (4) correspond to the case in which $\lambda \in \RR \backslash S_1$. It follows that each nonzero complex number satisfies exactly one of these conditions.
\begin{enumerate}
    \item If $\lambda$ is a primitive $n$th root of unity, then $\HS(\langle \lambda \rangle) = \HS(G)$ where $G$ is the group of $n$th roots of unity. The convex combinations of these are precisely the polygon $\Pi_n$.
    
    \item If $\lambda = e^{2\pi i \alpha}$, then $\HS(\langle \lambda \rangle) = \HS(G)$ where $G = \{e^{2\pi i \alpha n} \mid n \in \ZZ\}$. Since $\lambda$ is not a root of unity, this group is a countable dense subset of $S_1$. Thus, we find that $\HS(\langle \lambda \rangle)$ is the union of this countable dense subset together with the open unit ball, which is precisely $B_\alpha'$.
    
    \item If $\lambda$ is a positive real number not equal to 1, then powers of lambda can attain arbitrarily small and arbitrarily large positive real numbers. Thus, the set of convex combinations of these powers is precisely $\RR^+$.
    
    \item Similarly as above, if $\lambda$ is a negative real number not equal to $-1$, we may find powers of $\lambda$ which are arbitrarily large negative and positive real numbers. It therefore follows that $\HS(\langle \lambda \rangle) = \RR$ in this case.
    
    \item Otherwise, $\lambda$ is a complex number that does not lie on the unit circle and is not real. In this case, there are powers of $\lambda$ with arbitrarily large modulus. Since $\lambda$ is not real, these powers are not all contained in the same half-plane. Thus, their convex hull contains all complex numbers. 
\end{enumerate}
\end{proof}

\begin{definition}
We say that $\lambda \in \CC$ is \emph{of type 1} if it satisfies hypothesis 1 of the above lemma (i.e. if it is an integral root of unity). Similarly, we may classify any other nonzero complex number as type 1, 2, 3, 4, or 5 depending on which hypothesis it satisfies.
\end{definition}

\begin{lemma}
Let $G$ be a subgroup of the nonzero complex numbers, and let $\Lambda \subseteq G$ be any generating set of the group. Then
\begin{enumerate}
    \item if $\Lambda$ contains an element of type 5, then $\HS(G) = \CC$.
    
    \item if $\Lambda$ contains a non-real element of type 1 or 2 and an element of type 3 or 4, then $\HS(G) = \CC$.
    
    \item if $\Lambda \subseteq S_1$, then $\HS(G)$ is a union of some $\Pi_n$'s with some $B_\alpha'$'s.
    
    \item if $\Lambda \subseteq \RR^+$ but $\Lambda \not \subseteq S_1$, then $\HS(G) = \RR^+$
    
    \item if $\Lambda \subseteq \RR$ but $\Lambda \not \subseteq S_1$ and $\Lambda \not \subseteq \RR^+$, then $\HS(G) = \RR$
\end{enumerate}
\end{lemma}

\begin{proof}
We again address each case individually. Notice that if $\Lambda$ does not satisfy hypothesis 1 or 2, then $\Lambda$ is contained in either $S_1$ or $\RR$. The three remaining hypotheses therefore handle each remaining case, so any collection $\Lambda$ of nonzero complex numbers satisfies exactly one of these conditions.
\begin{enumerate}
    \item This case is trivial by Lemma \ref{SubgroupLemma}.
    
    \item Let $\lambda \in \Lambda$ be non-real and of type 1 or 2, and let $\lambda' \in \Lambda$ be of type 2 or 3. Then $\lambda \lambda' \in G$ is a non-real and has modulus not equal to 1, so we are again finished by Lemma \ref{SubgroupLemma}.
    
    \item Notice that if $\Lambda \subseteq S_1$ then $G \subseteq S_1$ as well. If $\Lambda$ is finite and all elements of $\Lambda$ are of type 1, then let $n$ be the least common multiple of all $k \in \NN$ such that $\Lambda$ contains a primitive $k$th root of unity. Then $G$ is precisely the $n$th roots of unity, so $\HS(G) = \Pi_n$ in this case. Otherwise $\Lambda$ is either infinite or contains an element of type 2, so $G$ is a dense subset of $S_1$. Thus $B_1 \subseteq \HS(G)$, so in fact $B_1 \cup G = \HS(G)$. It follows that $\HS(G) = \bigcup_{\lambda \in G} \HS(\langle \lambda \rangle)$, finishing our proof in this case.
    
    \item By Lemma \ref{SubgroupLemma} we know that $\RR^+ \subseteq \HS(G)$. Further, we have $G \subseteq \RR^+$ and $\HS(\RR^+) = \RR^+$, so the other direction also holds.
    
    \item We use the same reasoning as in (4).
\end{enumerate}
This completes the proof.
\end{proof}

With these lemmas in mind, we now have a full solution to the abelian group case as follows:

\begin{theorem}\label{AbelianTheorem}
Let $G$ be an abelian matrix group, and consider a generating set $\{A_i\}$ for $G$. let $\lambda_{i,1}, ..., \lambda_{i,n}$ be the eigenvalues of $A_i$ for each $i$. Then
\begin{align*}
    \HS(G) = \bigcup_{j=1}^n \HS(\langle \{\lambda_{i,j}\}_i \rangle)
\end{align*}
In particular, $\HS(G)$ is a union of five types of sets: $\CC, \RR, \RR^+, \Pi_k, $ and $B_\alpha '$
\end{theorem}

\begin{proof}
Since $G$ is a commuting family of matrices, we may choose a basis of $\CC^n$ wherein each $A_i^k \in G$ (for $k \in \ZZ$) is upper-triangular (see e.g.\ \cite{MatrixAnalysis}). Let $\lambda_{i,1}^k, ..., \lambda_{i,n}^k$ be the diagonal elements of each $A_i^k$. Then the eigenvalues of $\sum_{i,k} \alpha_{i,k} A_i^k$ are $\sum_{i,k} \alpha_k \lambda_{i,j}^k$ for $j=1,...,n$. It follows that the hull spectrum $\HS(G)$ is determined entirely by the eigenvalues (ordered according to the chosen basis) of the generators $A_i$.
\end{proof}

\begin{corollary}
For a finite abelian matrix group $G$, we have \[\HS(G) = \bigcup_{k \in \mc{K}} \Pi_k\] where $\mc{K}$ is the set of all $k \in \NN$ such that some element $A \in G$ has a primitive $k$th root of unity as an eigenvalue.
\end{corollary}

We conclude this section with a brief analysis of faithful hull spectra of abstract finite abelian groups.

\begin{proposition}\label{CyclicProp}
Let $\GG = \langle g \rangle$ be a cyclic group of finite order $n$. Then the faithful hull spectra of $\GG$ are of the form \[\bigcup_{j \in J} \Pi_j\] for any collection $J$ of positive integers whose least common multiple is $n$.
\end{proposition}

\begin{proof}
Note that the irreducible representations of $\GG$ are given by $\rho_k(g) = g^k$ for $k=0,1,...,n-1$, and $\HS(\rho_k(\GG)) = \Pi_{n/d}$ where $d=\gcd(n,k)$. Since a faithful representation of $\GG$ is a direct sum of some $\rho_k$ such that the $k$ have greatest common factor 1, it follows that the $n/d$ must have least common multiple $n$. Since any configuration of $\rho_k$ such that the $n/d$ have least common multiple $n$ is faithful, we are finished.
\end{proof}

\section{Bounds on the Hull Spectra of General Finite Groups}

Now, we consider $\HS(G)$ for general finite matrix groups $G$. A simple lower bound from our previous results is given by the following lemma.
\begin{lemma}\label{LowerBound}
For a finite matrix group $G$, we have $\bigcup_{k \in \mc{K}} \Pi_k \subseteq \HS(G)$, in which $\mc{K}$ is the set of all $k \in \NN$ such that some element $A \in G$ has a primitive $k$th root of unity as an eigenvalue. In particular, we have $\Pi_p \subseteq \HS(G)$ when $p$ is a prime that divides the order of $G$.
\end{lemma}
\begin{proof}
For any $A \in G$ with a primitive $k$th root of unity as an eigenvalue, we know by Theorem \ref{AbelianTheorem} that $\Pi_k \subseteq \HS(\langle A \rangle) \subseteq \HS(G)$.
\end{proof}

\begin{definition}
In view of Lemmas \ref{IsomorphicRepresentationLemma} and \ref{IrrepLemma}, we see that for a finite group $\GG$, the hull spectrum of a $\GG$-representation is entirely determined by which irreducible summands are present in its direct sum decomposition into irreducibles. Recall the \emph{left regular representation} of $\GG$, given by $\reg : \GG \to GL(\Vreg)$ where $\Vreg$ is the vector space of all functions $\GG \to \CC$ and $(\reg(g)f)(g') = f(g^{-1}g')$ for $f \in \Vreg$ and $g,g' \in \GG$.

It is a basic fact in representation theory (see e.g.\ \cite{Lang}) that the regular representation contains every irreducible as a summand, so $\HS(\reg(\GG))$ is an upper bound on the hull spectrum of any $\GG$-representation. We will refer to $\HS(\reg(\GG))$ as the \emph{maximal hull spectrum} of $\GG$. Note that in view of Proposition \ref{CyclicProp}, the maximal hull spectrum of the cyclic group of order $n$ is $\Pi_n$.
\end{definition}

\begin{lemma}\label{LowerBoundRegularRep}
For a finite group $\GG$, we have $\bigcup_{k \in \mc{K}} \Pi_k \subseteq \HS(\reg(\GG))$, in which $\mc{K}$ is the set of all $k \in \NN$ such that some element $g \in \GG$ has order $k$.
\end{lemma}
\begin{proof}
Suppose $g \in \GG$ has order $k$. By Lemma \ref{LowerBound}, it suffices to show that $\reg(g)$ has a primitive $k$th root of unity as an eigenvalue. Observe that each element of $\GG$ has an orbit of size $k$ under the action by $g$. Define $f \in \Vreg$ by mapping an arbitrary element of each orbit to $1$, and then subsequently mapping each other element to a $k$th root of unity so that $\reg(g)f = e^{2\pi i/k} f$. Then $e^{2\pi i/k}$ is an eigenvalue of $\reg(g)$, so we are finished.
\end{proof}

We know from Theorem \ref{AbelianTheorem} that finite abelian groups have hull spectra that are exactly unions of some polygons $\Pi_k$. In particular, the lower bounds given in Lemmas \ref{LowerBound} and \ref{LowerBoundRegularRep} are actually tight for finite abelian groups. We will later show that the lower bound is tight for certain other finite groups as well. However, the inclusions given by these lemmas are not tight in other cases (such as $\GG = S_5$ and many alternating groups $A_n$). Interestingly, they are known to be tight for $S_n$ with $n \leq 4$ and may be tight for many $S_n$ with $n > 5$. For these reasons, it will also be useful to determine some upper bounds on the hull spectra of groups.

To obtain an upper bound on $\HS(G)$ for a matrix group $G$, we use properties of the field of values of a matrix. For a matrix $A \in M_n$, the \textit{field of values} of $A$ is defined as
\[F(A) = \{x^* A x \mid x \in \CC^n, x^* x = 1\} \]
We make use of several fundamental properties of the field of values, as can be found in \cite{Topics}. For sets $X$ and $Y$ we denote their sum $X+Y = \{x + y \mid x \in X,\; y \in Y\}$.
\begin{lemma}\label{FieldOfValuesLemma}
For matrices $A, B \in M_n$,
\begin{itemize}
    \item $\sigma(A) \subseteq F(A)$
    \item $F(A+B) \subseteq F(A) + F(B)$
    \item $F(\alpha A) = \alpha F(A)$ for $\alpha \in \CC$
    \item If $A$ is normal, then $F(A) = \Co(\sigma(A))$
\end{itemize}
\end{lemma}

These properties immediately give an upper bound on $\HS(G)$.

\begin{lemma}\label{UpperBoundLemma}
For any finite matrix group $G$,
\[\HS(G) \subseteq \Co\left(\; \bigcup_{A \in G} \Co\big(\sigma(A)\big)\right)\]
\end{lemma}
\begin{proof}
We can assume all matrices in $G$ are unitary and thus also normal. Due to Lemma \ref{FieldOfValuesLemma}, for $\alpha_k \geq 0$ with $\sum_k \alpha_k = 1$, we have
\begin{align*}
\sigma \left( \sum_k \alpha_k A_k \right) & \subseteq F \left( \sum_k \alpha_k A_k \right) \\
        & \subseteq \sum_k F(\alpha_k A_k) \\
        & = \sum_k \alpha_k F(A_k) \\
        & = \sum_k \alpha_k \Co \big( \sigma(A_k) \big) \\
        & \subseteq \Co \left( \bigcup_k \Co \big( \sigma(A_k) \big) \right),
\end{align*}
so the claim holds.
\end{proof}

Together with the previously-derived lower bound on $\HS(G)$, this upper bound yields the following results:

\begin{corollary}\label{TotalBoundsCorollary}
For a finite matrix group $G$, we have \[ \bigcup_{k \in \mc{K}} \Pi_k \subseteq \HS(G) \subseteq \Co \left( \bigcup_{k \in \mc{K}} \Pi_k \right) \] in which $\mc{K}$ is the set of $k$ such that some $A \in G$ has a primitive $k$th root of unity as an eigenvalue.
\end{corollary}

\begin{corollary}\label{TotalRegRepBoundsCorollary}
For a finite group $\GG$, we have \[ \bigcup_{k \in \mc{K}} \Pi_k \subseteq \HS(\reg(\GG)) \subseteq \Co \left( \bigcup_{k \in \mc{K}} \Pi_k \right) \] in which $\mc{K}$ is the set of $k$ such that some $g \in \GG$ has order $k$.
\end{corollary}

Recall the doubly stochastic matrix
\[C_t = \begin{bmatrix}
0 & 0 & 0 & 1 & 0 \\
0 & 0 & t & 0 & 1-t \\
0 & t & 1-t & 0 & 0 \\
0 & 1-t & 0 & 0 & t \\
1 & 0 & 0 & 0 & 0
\end{bmatrix},\]
which for $t=\frac{1}{2}$ serves as the example in \cite{MR} that shows that $DS_5 \neq \bigcup_{k \leq 5} \Pi_k$. We remark that the eigenvalues realized by $C_t$ for any $t \in [0,1]$ all lie within $\bigcup_{k \leq 6} \Pi_k$, the lower bound given in Corollary \ref{TotalRegRepBoundsCorollary} (since $S_5$ has elements of orders $1$ through $6$).  As we have seen in section \ref{AbelianSection}, $\HS(\reg(\GG))$ achieves the lower bound when $\GG$ is abelian. The following two sections prove that this is also the case for the dihedral groups and the quaternion group. However, in Section \ref{sec:7} we present an example in $A_4$ for which the lower bound in Corollary \ref{TotalRegRepBoundsCorollary} is not achieved.

\section{Dihedral Groups}

In this section we consider the dihedral group \[D_{2n} = \langle r,s \mid r^n = s^2 = (sr)^2 = 1 \rangle\] of symmetries of the $n$-gon. The following two observations will classify the irreducible representations of $D_{2n}$ (see \cite{DihedralIrreps}) so that we may analyze the resulting hull spectra.

\begin{observation}\label{DihedralOddObservation}
We note that if $n$ is odd, there are $\frac{n+3}{2}$ irreducible representations of $D_{2n}$. Two of these representations are one-dimensional, so their images are abelian groups of complex numbers. In particular, the only two groups realized are the trivial group $\{1\}$ and $\ZZ/2\ZZ = \{-1, 1\}$.

The remaining $\frac{n-1}{2}$ irreducibles are of the form $\rho_k : D_{2n} \to GL(\CC^2)$ in which 
\begin{align*}
    \rho_k(r) = \begin{bmatrix} \cos k\theta_n & -\sin k\theta_n \\ \sin k\theta_n & \cos k\theta_n \end{bmatrix}
    \quad \quad \quad
    \rho_k(s) = \begin{bmatrix} 0 & 1 \\ 1 & 0 \end{bmatrix}
\end{align*}
for $\theta_n = \frac{2\pi}{n}$ and $k=1,...,\frac{n-1}{2}$.
\end{observation}

\begin{observation}\label{DihedralEvenObservation}
If $n$ is even, there are $\frac{n+6}{2}$ irreducible representations of $D_{2m}$, four of which are one-dimensional with images $\{1\}$ or $\{-1, 1\}$. As above, the rest are of the form $\rho_k : D_{2n} \to GL_2(\CC)$ in which 
\begin{align*}
    \rho_k(r) = \begin{bmatrix} \cos k\theta_n & -\sin k\theta_n \\ \sin k\theta_n & \cos k\theta_n \end{bmatrix}
    \quad \quad \quad
    \rho_k(s) = \begin{bmatrix} 0 & 1 \\ 1 & 0 \end{bmatrix}
\end{align*}
for $\theta_n = \frac{2\pi}{n}$ and $k=1,...,\frac{n-2}{2}$.
\end{observation}

\begin{theorem}\label{DihedralIrrepTheorem}
With $\rho_k$ defined as above, \[\HS(\rho_k(D_{2n})) = \Pi_a \cup \Pi_2\] for $a = \frac{n}{\gcd(n,k)}$.
\end{theorem}

\begin{proof}
Since $\rho_k(D_{2n}) = \rho_{\frac{k}{\gcd(n,k)}}(D_{2a})$, we may reduce to the case in which $n=a$ and $k = \frac{k}{\gcd(n,k)}$, so $k$ is prime to $n$. Under this assumption, we see that the elements of $\rho_k(D_{2n})$ are exactly the same as those in $\rho_1(D_{2n})$. Thus, we may assume without loss of generality that $k=1$. We now set $G = \rho_1(D_{2n})$ and aim to prove $\HS(G) = \Pi_n \cup \Pi_2$.

The inclusion $\HS(G) \supseteq \Pi_n \cup \Pi_2$ is trivial by Lemma \ref{SubgroupLemma}, as $\HS(\langle r \rangle) = \Pi_n$ and $\HS(\langle s \rangle) = \Pi_2$. We now claim that the spectrum of every element of $\Co(G)$ is contained within $\Pi_n \cup \Pi_2$. Let $C \in \Co(G)$, so it is of the following form:
\begin{align*}
		C & = \alpha_0 I +  \alpha_1 r + \ldots + \alpha_{n-1} r^{n-1} + \beta_0 s + \beta_1 rs + \ldots + \beta_{n-1}r^{n-1}s\\ 
& = \begin{bmatrix}
				\sum_{k=0}^{n-1} \left( \alpha_k \cos(k\theta_n) - \beta_k \sin(k \theta_n) \right) & \beta_0 + \sum_{k=0}^{n-1} \left( -\alpha_k \sin(k \theta_n) + \beta_k \cos(k \theta_n) \right) \\
				\beta_0 + \sum_{k=0}^{n-1} \left( \alpha_k \sin(k \theta_n) + \beta_k \cos(k \theta_n) \right) & \sum_{k=0}^{n-1} \left( \alpha_k \cos(k \theta_n) + \beta_k \sin(k \theta_n) \right)
\end{bmatrix}
\end{align*}
where $\alpha_k \geq 0$, $\beta_k \geq 0$, $\sum_k ( \alpha_k + \beta_k ) = 1$ and $\theta_n = \frac{2\pi}{n}$. Note that by Lemma \ref{UnitBallLemma}, we know that $\HS(G) \subseteq \ol{B}_1$. Then if $C$ has real eigenvalues, it follows that $\sigma(C) \subseteq \Pi_2$.

Otherwise $C$ has non-real eigenvalues $\lambda, \ol{\lambda}$ for some $\lambda \in \CC \backslash \RR$. Then we see that
\begin{align*}
		\mrm{Re}(\lambda) = \frac{1}{2}\Tr(C) & = \sum_{k=0}^{n-1}\alpha_k\cos(k \theta_n)
\end{align*}
Let $B = \frac{1}{2i}(C - C^*)$. It is well-known (see \cite{BoundsForEigenvalues}) that $\lambda_{min}(B) \leq \mrm{Im}(\lambda) \leq \lambda_{max}(B)$. Also,
\[B = \begin{bmatrix}
				0 & \frac{-1}{i}\sum_{k=0}^{n-1} \alpha_k \sin(k \theta_n)\\
				\frac{1}{i}\sum_{k=0}^{n-1} \alpha_k \sin(k \theta_n) & 0
\end{bmatrix}\]
has eigenvalues $\pm \sum_{k=0}^{n-1} \alpha_k \sin(k \theta_n)$. Thus, the eigenvalues of $C$ have real part $\sum_{k}\alpha_k \cos(k \theta_n)$ for $\alpha_k \geq 0$ and $\sum_k \alpha_k \leq 1$, and they have imaginary part of magnitude at most $\sum_k \alpha_k \sin(k \theta_n)$. Since each $\cos(k\theta_n) \pm i \sin (k\theta_n)$ lies in $\Pi_n$, we see that $\lambda$ is a convex combination of points in $\Pi_n$, scaled by the constant $\sum_k \alpha_k \in [0,1]$ and with imaginary part also scaled by some constant in $[0,1]$. Hence $\lambda, \ol{\lambda} \in \Pi_n$, so that the proof is complete.
\end{proof}

\begin{corollary}\label{DihedralFaithfulCorollary}
The faithful hull spectra of $D_{2n}$ are given by \[\Pi_2 \cup \bigcup_{j \in J} \Pi_j\] for any collection $J$ of positive integers whose least common multiple is $n$. The maximal hull spectrum is then \[\Pi_2 \cup \Pi_n\]
\end{corollary}

\begin{proof}
By Lemma \ref{LowerBound}, we know that any faithful hull spectrum of $D_{2n}$ contains $\Pi_2 \cup \bigcup_{p} \Pi_p$ where $p$ ranges over the prime factors of $n$. In fact, this hull spectrum is realizable by taking the direct sum over representations $\rho_k$ such that $n/\gcd(n,k)$ is an odd prime.

Now observe that whenever $m>2$ divides $n$, the hull spectrum $\Pi_m \cup \Pi_2$ is realizable by $\rho_{n/m}$. Since we may consider the direct sum of as many irreducible representations as we want, and since no other hull spectrum is realizable by any irreducible representation, our result follows.
\end{proof}

\section{The Quaternion Group}

We now shift our attention to the quaternion group
\[ Q_8 = \langle r,s \mid r^4=1,\; s^2=r^2,\; s^{-1}rs = r^{-1} \rangle. \]
Note that every element of $Q_8$ can be written as $r^k s^j$ for $k=0,1,2,3$ and $j=0,1$. Following the place of the previous section, we will first examine the irreducible representations of $Q_8$ (see \cite{DicyclicIrreps}).

\begin{observation}\label{DicyclicEvenObservation}
Notice that $D_4 = Q_8/Z$ where $Z=\{1,r^2\}$ is the center of $Q_8$, so many $Q_8$-representations factor through $D_4 = \ZZ/2\ZZ \times \ZZ/2\ZZ$. In particular, there are four such irreducibles, each of degree 1 (since $D_4$ is abelian). These irreducibles all have image contained in $\{-1, 1\}$ and map $r^2=s^2$ to 1. The remaining irreducible is $\rho : Q_8 \to GL_2(\CC)$ given by
\begin{align*}
    \rho(r) = \begin{bmatrix} i & 0 \\ 0 & -i \end{bmatrix}
    \quad \quad \quad
    \rho(s) = \begin{bmatrix} 0 & -1 \\ 1 & 0 \end{bmatrix}
\end{align*}
\end{observation}

\begin{theorem}
With $\rho$ defined as above, we have
\[\HS(\rho(Q_8)) = \Pi_4\]
\end{theorem}

\begin{proof}
Here, we write $r$ to refer to the matrix $\rho(r)$ and $s$ to refer to $\rho(s)$. As before, the inclusion $\HS(\rho(Q_8)) \supseteq \Pi_4$ is trivial by Lemma \ref{SubgroupLemma}, since $\HS(\langle r \rangle) = \Pi_4$. Now let $C \in \Co(G)$, so we may write
\begin{align*}
    C &= \sum_{k=0}^3 \alpha_j r^k + \sum_{k=0}^3 \beta_j r^ks \\
    &= \begin{bmatrix}
    \sum_{k=0}^3 \alpha_k i^k & -\sum_{k=0}^3 \beta_k i^k \\
    \sum_{k=0}^3 \beta_k (-i)^k & \sum_{k=0}^3 \alpha_k (-i)^k
    \end{bmatrix}
\end{align*}
for some $\alpha_j, \beta_j \geq 0$ satisfying $\sum_j (\alpha_j + \beta_j) = 1$. Then with $\alpha_j' = \alpha_j-\alpha_{j+2}$ and $\beta_j' = \beta_j-\beta_{j+2}$, we get
\begin{align*}
    C &= \begin{bmatrix}
    \alpha_0' + \alpha_1' i & -\beta_0' - \beta_1' i \\
    \beta_0' - \beta_1' i & \alpha_0' - \alpha_1' i
    \end{bmatrix},
\end{align*}
so the eigenvalues of $C$ are $\alpha_0' \pm i \sqrt{(\alpha_1')^2 + (\beta_0')^2 + (\beta_1')^2}$. Thus, the imaginary part of each eigenvalue $\lambda$ has magnitude at most $|\alpha_1'| + |\beta_0'| + |\beta_1'|$, and the real part has magnitude $|\alpha_0|$. Then
\begin{align*}
    |\text{Re}\; \lambda| + |\text{Im}\; \lambda| &\leq |\alpha_0| + |\alpha_0'| + |\beta_0'| + |\beta_1'| \\
    &\leq \sum_{k=0}^3 (\alpha_j + \beta_j) \\
    &= 1,
\end{align*}
so indeed $\lambda \in \Pi_4$, as desired. 
\end{proof}

\begin{corollary}
The only faithful hull spectrum (and maximal hull spectrum) of $Q_8$ is $\Pi_4$.
\end{corollary}

\begin{proof}
Recall that $r^2=s^2$ belongs to the kernel of each degree-1 representation of $Q_8$, so any faithful representation must contain a copy of $\rho$. Thus, the corollary follows by the preceding theorem.
\end{proof}

\section{Doubly Stochastic Matrices and the Symmetric Group} \label{sec:7}

Work thus far in trying to understand $DS_n$ has been matricial. The permutation matrix representation of $S_n$ is a direct sum of the trivial representation and the standard representation. Thus, by Lemma \ref{IrrepLemma}, $DS_n$ is also the set of all eigenvalues achievable by convex combinations of matrices in the standard representation of $S_n$, which is of dimension $n-1$. Here, we note some known facts about $DS_n$ that may be proven via the group structure of $S_n$, perhaps more easily than using other proof methods. This analysis shows that some of the theory used to investigate $DS_n$ (such as the Perron-Frobenius theory of nonnegative matrices and assorted facts about doubly-stochastic matrices) may not be necessary for deducing information about the region.

\begin{proposition}\label{DS_nProps}
For $n \geq 2$, the region $DS_n$:
\begin{enumerate}
    \item Is star-shaped from any point in $[0,1]$
    \item Contains the polygons $\Pi_k$ for $k \leq n$
    \item Is contained in the convex hull of the polygons $\Pi_k$ for $k \leq n$ and thus within the unit disc.
\end{enumerate}
Moreover, for any doubly stochastic matrix $A \in M_n$ whose eigenvalues all lie on the unit circle, $A$ is a permutation matrix.
\end{proposition}
\begin{proof}
Property (1) follows from Corollary \ref{StarCorollary}, and (2) and (3) both follow from Corollary \ref{TotalBoundsCorollary}. The last component is due to Proposition \ref{ExtremalEigenvalues}.
\end{proof}

Moreover, we can apply our results on the dihedral group to provide a new proof of the form of $DS_3$.

\begin{proposition}
$DS_3 = \Pi_2 \cup \Pi_3$
\end{proposition}
\begin{proof}
Since $S_3 \cong D_6$ and since the permutation matrix representation of $S_3$ is faithful, Corollary \ref{DihedralFaithfulCorollary} proves the result.
\end{proof}

The upper bound on $DS_n$ given by the field of values bound in Proposition \ref{DS_nProps} does not add information because of the work of Karpelevich in \cite{Karpelevich}. The region $K_n$ of eigenvalues achieved by row stochastic matrices is an outer estimate of $DS_n$, and is in fact strictly smaller than the field of values bound for $n \geq 3$. The boundary of $K_n$ consists of concave arcs between adjacent roots of unity, and these arcs lie within the line segments between adjacent roots of unity given by the field of values bound.

Again, we note that although all finite abelian groups, dihedral groups, the quaternion group, and the standard representation of $S_n$ for $n \leq 4$ have hull spectra corresponding to the polygon lower bound of Lemma \ref{LowerBound}, the standard representation of $S_5$ does not. This phenomenon makes the question of determining $DS_n$ all the more compelling. Here, we note that the standard representation of the alternating group $A_n$ also does not in general agree with the lower bound of Lemma \ref{LowerBound}. First, we note the form of the lower bound:

\begin{lemma}
The hull spectrum of the standard representation $\rho$ of $A_n$ contains $\Pi_k$ for all odd $k \leq n$ and all even $k \leq n-2$.
\end{lemma}
\begin{proof}
For all odd $k \leq n$, $A_n$ contains any k-cycle. For even $k \leq n-2$, $A_n$ contains the product of any k-cycle with a disjoint 2-cycle.
\end{proof}

Now it may be seen through direct computation that $\HS(\rho(A_n))$ is not contained within the polygon lower bound for some small $n > 3$. For example, the standard representation $\rho$ of $A_4$ has its hull spectrum bounded below by $\Pi_2 \cup \Pi_3$. Consider the convex combinations of the even permutation matrices corresponding to the permutations $(123)$ and $(12)(34)$:

\[C_t = (1-t) \begin{bmatrix}
0 & 1 & 0 & 0\\
0 & 0 & 1 & 0\\
1 & 0 & 0 & 0\\
0 & 0 & 0 & 1\\
\end{bmatrix} + t \begin{bmatrix}
0 & 1 & 0 & 0\\
1 & 0 & 0 & 0\\
0 & 0 & 0 & 1\\
0 & 0 & 1 & 0\\
\end{bmatrix}\]

Note that $C_t$ has a degree-$4$ characteristic polynomial that can be factored into $(x-1)q(x)$ for a degree $3$ polynomial $q(x) \in \RR[x]$, so its roots can be determined with standard techniques. For $t\in [.4,.9]$ (this range is not tight), $C_t$ has a non-real eigenvalue $\lambda$ with $\mrm{Re}\; \lambda < -.5$, which means that it is to the left of $\Pi_3$ and does not belong to $\Pi_2$. Thus, there is a large portion of the hull spectrum of $\rho(A_4)$ that lies outside of the polygon lower bounds.

\section*{Declaration of competing interest}
None.

\section*{Acknowledgements}

We would like to thank Jamison Barsotti for bringing our attention to several useful resources that we were unaware of. We also wish to thank the anonymous referee for carefully reading our work and suggesting changes that have improved this paper. This work was supported by the National Science Foundation grant DMS \#1757603.

\bibliographystyle{ieeetr}
\bibliography{refs}

\end{document}